\documentclass[12pt]{amsart}
\usepackage[mathscr]{eucal}
\usepackage{graphicx}

\usepackage{amsmath,amssymb,color}

\DeclareMathOperator{\sech}{sech}

\usepackage{tabls}  

\usepackage{mathtools}
\DeclarePairedDelimiter{\abs}{\lvert}{\rvert}

\newtheorem{thm}{Theorem}[section]
\newtheorem*{thm*}{Theorem}
\newtheorem{lemma}[thm]{Lemma}

\theoremstyle{definition}
\newtheorem{defn}[thm]{Definition}
\newtheorem{ex}{Example}
\newtheorem{rem}{Remark}

\begin{document}


\newcommand{\marginnote}[1]{\mbox{}\marginpar{\raggedleft\hspace{0pt}#1}}

%
%
\newcommand{\R}{\mathbb R}
\newcommand{\Z}{\mathbb Z}

\newcommand{\Rplus}{\R^+}
\newcommand{\Rnz}{\R^*}

%
%

\newcommand{\rto}{\R^{2,1}}

%
%
\newcommand{\A}{\mathsf A} 

\newcommand{\E}{\mathsf E} 
\newcommand{\V}{\mathsf V} 

\newcommand{\Ht}{\mathsf{H}^2} 

\newcommand{\Fut}{\mathsf{Future}}

%
%
\renewcommand{\o}{\operatorname}

\newcommand{\Oto}{\o{O}(2,1)}
\newcommand{\SOto}{\o{SO}(2,1)}
\newcommand{\SOoto}{\o{SO}^{0}(2,1)}

\newcommand{\SOoo}{\o{SO}^{0}(1,1)}

\newcommand{\GL}{\o{GL}}
\newcommand{\GLp}{\GL^+}

\newcommand{\PSLtR}{\o{PSL}(2,\R)}
\newcommand{\SLtR}{\o{SL}(2,\R)}

\newcommand{\CONF}{{\operatorname{Conf}(\Eint)}}

\newcommand{\Aff}{\o{Aff}^+}
\newcommand{\Conf}{\o{Conf}^+}
\newcommand{\Isom}{\o{Isom}^+}
\newcommand{\Confo}{\o{Conf}^0}

\newcommand{\Ise}{\Isom(\E)}

\newcommand{\IsE}{\o{Isom}^+(\E)}
\newcommand{\ConfE}{\o{Conf}^+(\E)}

%
%
\newcommand{\slr}{\mathfrak{sl}(2)}
\newcommand{\soto}{\mathfrak{so}(2,1)}

%
%
\newcommand{\h}{\mathfrak{h}}     

%
%
\newcommand{\C}{{\mathcal C}} 
\renewcommand{\H}{\mathcal{H}} 
\newcommand{\CP}{\C}
\newcommand{\CHS}[1]{\H(#1)}
\newcommand{\WP}{{\mathcal W}}
\newcommand{\TS}{{\mathcal T}}
\renewcommand{\SS}{{\mathcal S}}

\renewcommand{\lg}{l_\gamma}

%
%
\newcommand{\stem}[1]{{\operatorname{Stem}}(#1)}
\newcommand{\wing}[1]{{\operatorname{Wing}}(#1)}

%
%
\newcommand{\Quad}{\mathsf{V}_{\scriptscriptstyle\neq\origin} }
\newcommand{\allow}[1]{{\mathsf A}(#1)}  

%
%
\newcommand{\Gg}{G} 
\newcommand{\Go}{G_0} 

%
%
\newcommand{\g}{g} 
\newcommand{\ag}{\gamma} 

%
%
\newcommand{\vu}{\mathsf{u}}
\newcommand{\vv}{\mathsf{v}}
\newcommand{\vw}{\mathsf{w}}
\newcommand{\vs}{{\mathsf{s}}}
\newcommand{\vt}{{\mathsf{t}}}
\newcommand{\vn}{{\mathsf{n}}}
\newcommand{\va}{{\mathsf{a}}}
\newcommand{\vb}{{\mathsf{b}}}
\newcommand{\vc}{{\mathsf{c}}}
\newcommand{\vx}{{\mathsf{x}}}
\newcommand{\vy}{{\mathsf{y}}}
\newcommand{\vz}{\mathsf{z}}
\newcommand{\ve}{{\mathsf{e}}}
\newcommand{\vp}{{\mathsf{p}}}

\newcommand{\origin}{\mathsf{0}}

%
%
\newcommand{\dv}{\vu} 
\newcommand{\tp}{\vv} 
\newcommand{\logg}{\vx} 

%
%
\newcommand{\po}{o}

%
%
\newcommand{\xo}[1]{#1^0}
\newcommand{\xp}[1]{#1^+}
\newcommand{\xm}[1]{#1^-}
\newcommand{\xpm}[1]{#1^{\pm}}

\newcommand{\Fix}{\mathsf{Fix}}

%
%
\newcommand{\lx}{X}
\newcommand{\ly}{Y}
\newcommand{\lw}{W}
\newcommand{\lbasis}{E}
\newcommand{\lb}{B}
\newcommand{\lf}{F}

%
%
\newcommand{\conf}[1]{\overline{#1}^{\mbox{\tiny conf}}}

\newcommand{\Det}{\mathsf{Det}}

\renewcommand{\L}{\mathsf{L}}  

\newcommand{\ldot}[2]{#1 \cdot #2 }
\newcommand{\lcross}[2]{#1 \times #2 }

\newcommand{\newalpha}{\tilde{\alpha}}

\newcommand{\sgn}{\mathsf{sgn}}

%
%
\newcommand{\surf}{\Sigma}

\newcommand{\inte}[1]{\operatorname{int}\left(#1\right)}
\newcommand{\cl}[1]{\operatorname{cl}\left(#1\right)}

\title[Crooked foliations]
{Foliations of Minkowski $2+1$ spacetime by crooked planes}

\author[Charette]{Virginie Charette}
    \address{{\it Charette:\/}
    D\'epartement de math\'ematiques\\ Universit\'e de Sherbrooke\\
Sherbrooke,  Qu\'ebec J1K 2R1  Canada}
    \email{v.charette@usherbrooke.ca}
\author[Kim]{Youngju Kim}
  \address{ {\it Kim:\/}
   Korea Institute for Advanced Study\\ Heogirho 85, Dongdaemun-gu Seoul, 130-722 Korea}
   \email{geometer1@kias.re.kr}

\date{\today }

\begin{abstract}
Given a regular curve in Minkowski spacetime, we describe necessary and sufficient conditions for this curve to admit a family of pairwise-disjoint crooked planes.  Using this criterion, we describe crooked foliations along orbit curves of one-parameter groups of Lorentzian isometries.
\end{abstract}

\maketitle

\tableofcontents

Crooked planes are piecewise-linear surfaces in Minkowski spacetime; they were introduced by Drumm to construct fundamental domains for actions of free groups of Lorentzian isometries~\cite{D92}.   They have become quite useful for studying proper actions of free groups in Minkowski spacetime and other spaces, see for instance~\cite{CDG10,CDG11,Go13}.  Goldman asked the following question\,: given a pair of disjoint crooked planes, when can the space between them be foliated by pairwise disjoint crooked planes? The first steps in this direction may be found in~\cite{BCDG}.  In the present note, we construct foliations of Minkowski spacetime by crooked planes, along certain orbits of one-parameter subgroups of isometries.

In the process, we have sharpened the results of~\cite{BCDG} to give a necessary and sufficient condition for a curve in Minkowski spacetime to admit a foliation by crooked planes.  This condition is infinitesimal : it expresses the so-called ``Drumm-Goldman inequality'' for a pair of crooked planes in terms of the derivative of the curve.  In brief, crooked planes along a sufficiently smooth curve are pairwise disjoint if and only if the derivative of the curve belongs to the ``stem quadrant'' of the crooked plane at that point.  These terms will be defined in \S\ref{sec:fol}.

The condition, proved in \S\ref{sec:inf}, is interesting in its own right, since it can be applied to an arbitrary, sufficiently smooth curve. As far as this paper is concerned, we will examine orbit curves of one-parameter groups of Lorentzian isometries.  Specifically, we consider groups of the form $\langle g_t\rangle$, where $g_t$ is the exponential curve of a direction in the Lie algebra of the group of Lorentzian isometries.  The upshot is that all the elements share the same eigenspace, facilitating explicit calculations.  The orbit curves in question will be $\langle g_t\rangle$-orbits of points and we will consider crooked planes whose ``directors'' lie in a particular orbit of the linear parts of $\langle g_t\rangle$.  We consider the case where $g_t$ is hyperbolic as well as parabolic.

In the last section, we apply this to give sufficient criteria for a pair of disjoint crooked planes with ultraparallel directors to admit a foliation between them.  Unfortunately, we cannot provide a full picture for now.  Indeed, we can show that the pair does admit a foliation along as it is {\em calibrated}; we will explain this term in~\S\ref{sec:when}.  It is quite clear that our result can be generalized, for instance by relaxing the conditions on the directors of the crooked planes.  This is work in progress.

\section{Preliminaries}

Let $\V$ denote $\R^3$ endowed with a scalar product of signature $(2,1)$.  To fix ideas, we will usually assume that it takes the following form in the standard basis\,:
\begin{equation*}
\begin{bmatrix}x_1\\ x_2\\ x_3\end{bmatrix}\cdot\begin{bmatrix}y_1\\ y_2\\ y_3\end{bmatrix}=x_1y_1+x_2y_2-x_3y_3.
\end{equation*}
A vector $\vv\neq\origin\in\V$ is called
\begin{itemize}
 \item \emph{timelike} if $\ldot{\vv}{\vv} < 0$,
 \item \emph{null} (or {\em lightlike\/})
 if $\ldot{\vv}{\vv} = 0$,
\item \emph{spacelike} if $\ldot{\vv}{\vv}> 0$; when $\ldot{\vv}{\vv}=1$, it is called {\em unit-spacelike}.
\end{itemize}
The set of null vectors is called the {\em lightcone}.

Say that vectors $\vu,\vv\in\V$ are
{\em Lorentz-orthogonal\/} if $\vu \cdot \vv = 0.$
Denote the linear subspace of vectors Lorentz-orthogonal
to $\vv$ by $\vv^{\perp}$.

Let $\E$ be the affine space modeled on $\V$.  The vector space $\V$, considered as a Lie group, acts transitively on $\E$ by translations as follows\,:
\begin{align*}
\V\times\E & \longrightarrow\E \\
\left(\vt,p\right) & \longmapsto p+\vt.
\end{align*}

Setting $\po=(0,0,0)$, we can write any $p\in\E$ in terms of the $\V$-action on $\E$ by translation\,:
\begin{equation*}
p=\po+\vp
\end{equation*}
for $\vp \in \V$ and we use this to extend the action of a linear map $g:\V\rightarrow\V$ to an affine one\,:
\begin{equation*}
g(p):=\po+g(\vp).
\end{equation*}
Any affine map $\ag$ can be written as $t_\tp\circ g$, where $g$ is linear and $t_\tp$ is translation by the vector $\tp\in\V$.  We call $g$ the {\em linear part} of $\ag$ and denote it by $\L(\ag)$; we call $\tp$ the {\em translational part} of $\ag$.

Let us denote by $\Ise$ the group of orientation-preserving affine isometries that preserve the scalar product.   The linear part of an element of $\Ise$ belongs to $\SOto$, which is isomorphic to the group of isometries of the hyperbolic plane.  In keeping with the terminology of hyperbolic geometry, $g\in\SOto$ is called\,:
\begin{itemize}
\item {\em hyperbolic} if it has three distinct real eigenvalues;
\item {\em parabolic} if 1 is its only eigenvalue;
\item {\em elliptic} otherwise.
\end{itemize}
More generally, $g\in\Ise$ is called, hyperbolic, parabolic or elliptic, according to the nature of its linear part.

Identifying $\Ise$ with the subgroup of $GL(4)$ consisting of matrices of the form\,:
\begin{equation*}
\gamma=\begin{bmatrix} A & \tp \\ 0 & 1\end{bmatrix}
\end{equation*}
where $A\in\SOto$, the Lie algebra of $\Ise$ consists of matrices of the form\,:
\begin{equation*}
X=\begin{bmatrix} \logg & \vy \\ 0 & 0\end{bmatrix}
\end{equation*}
where $\logg\in\soto$. Choose $X$ as above, with $\logg\neq\origin$.  It generates a rank one sub-algebra, and so there is a unique rank one subgroup of $\Ise$, $\langle \ag_t\rangle$, such that $\ag_t=\exp(tX)$.  In fact\,:
\begin{equation*}
\L(\gamma_t)=\exp(t\logg).
\end{equation*}
The linear parts share a common eigensystem (for $t\neq 0$).

\section{Crooked planes and foliations}\label{sec:fol}

Let $\vu\in\V$ be a spacelike vector.  Since $\vu^\perp$ intersects the light cone in a pair of lines, we can form a basis for $\V$ containing $\vu$ and a vector spanning each of these lines.

\begin{defn}\label{def:nullframe}
Let $\vu\in\V$ be spacelike.  The {\em null frame} associated to $\vu$ is the basis $\left(\vu,\xm{\vu},\xp{\vu}\right)$, where $\xpm{\vu}$ are null vectors such that\,:
\begin{itemize}
\item $\vu^\perp=\langle\xm{\vu},\xp{\vu}\rangle$;
\item the third coordinate is 1;
\item $\det\left[ \vu~\xm{\vu}~\xp{\vu}\right]>0$.
\end{itemize}
\end{defn}
The second requirement is simply a useful normalizing condition; other similar conditions appear elsewhere~\cite{BCDG, DG99}.  Alternatively, $\xpm{\vu}$ is ``future-pointing''.

\begin{defn}
Let $p\in\E$ and $\vu\in\V$ be spacelike.  The {\em crooked plane} with {\em vertex} $p$ and {\em director} $\vu$ is the
union of\,:
\begin{itemize}
\item a {\em stem}\,:
\begin{equation*}
p+\{ \vx\in\vu^\perp:\vx\cdot\vx\leq 0\}
\end{equation*}
\item and a pair of  {\em wings}\,:
\begin{align*}
& p+\{ \vx\in\left(\xp{\vu}\right)^\perp:\xp{\vx}=\xp{\vu}\} \\
& p+\{ \vx\in\left(\xm{\vu}\right)^\perp:\xp{\vx}=\xm{\vu}\}.
\end{align*}
\end{itemize}
It is denoted $\CP(p,\vu)$.
\end{defn}

Figure~\ref{fig:cp} shows a crooked plane with director $\begin{bmatrix} 1\\ 0 \\ 0\end{bmatrix}$ and Figure~\ref{fig:pair} shows a pair of disjoint crooked planes.
 \begin{figure}
  \includegraphics{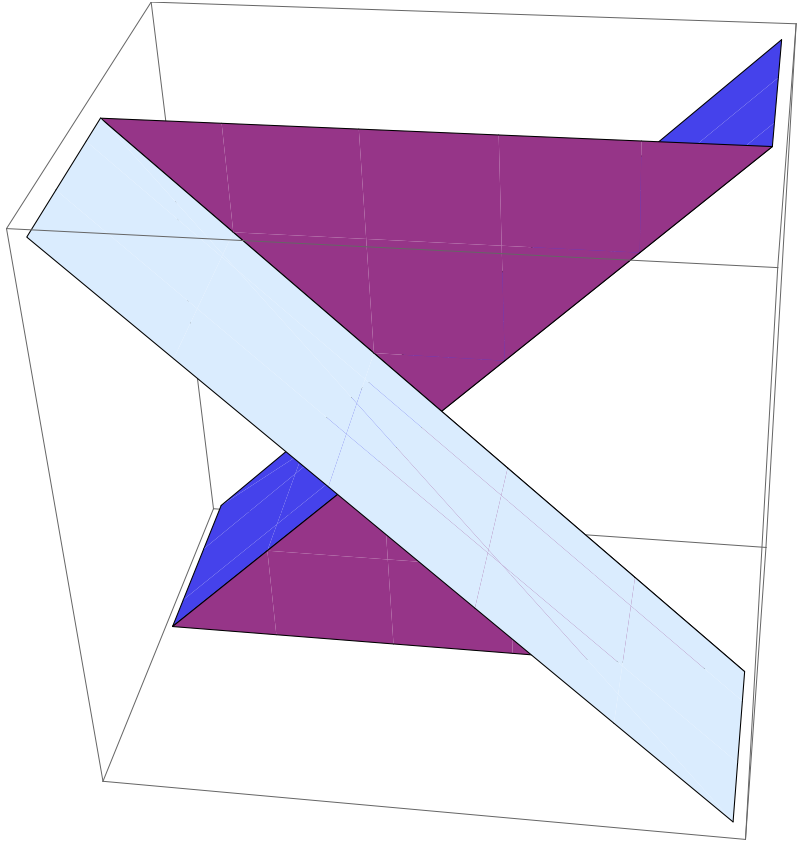}
  \caption{A crooked plane.}
  \label{fig:cp}
  \end{figure}
 \begin{figure}
  \includegraphics{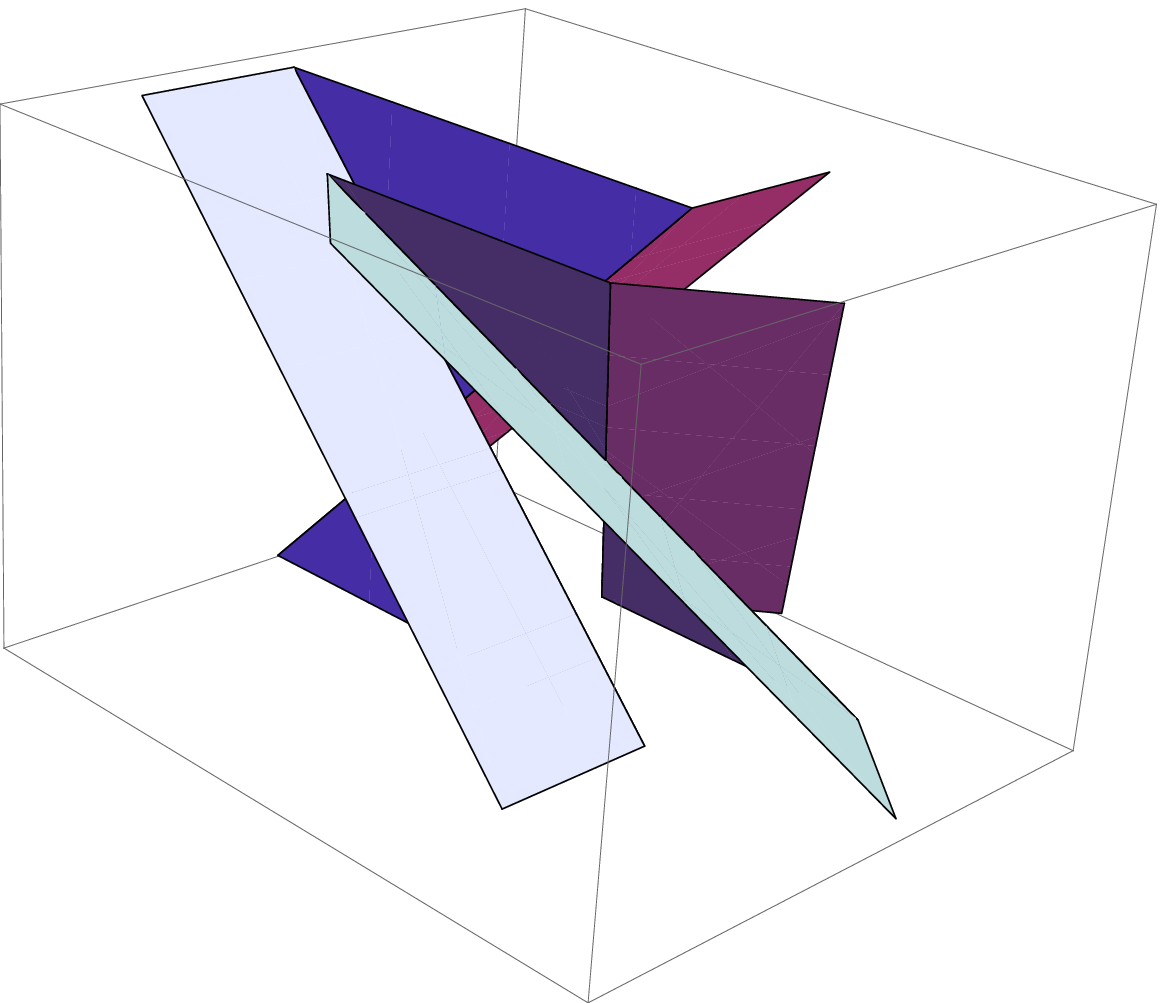}
  \caption{A pair of disjoint crooked planes.}
  \label{fig:pair}
  \end{figure}

We will now write down the condition for a pair of crooked planes to be disjoint.  First of all, the directors must be {\em non-crossing}, a term we shall explain here.  A pair of spacelike vectors $\vu_1,\vu_2\in\V$ are said to be {\em crossing} (respectively, {\em ultraparallel}, {\em asymptotic}) if $\vu_1^\perp\cap\vu_2^\perp$ is timelike (respectively, spacelike, null).  A non-crossing pair is either ultraparallel or asymptotic.  

We further require some normalizing conditions on the directors.  We say that a pair of non-crossing spacelike vectors $\vu_1,\vu_2\in\V$ are {\em consistently oriented} if\,:
\begin{itemize}
\item $\ldot{\vu_1}{\vu_2}<0$;
\item $\ldot{\vu_i}{\xpm{\vu_j}}\leq 0$ for $i,j=1,2$.
\end{itemize}
Given two spacelike, non-crossing vectors $\vw_1,\vw_2$, there is a unique choice of unit-spacelike vector $\vu_i\in\R\vw_i$ such that $\vu_1,\vu_2$ are consistently oriented.  In other words, there is a unique choice of pair of consistently oriented directions.

Now let $\CP(p_1, \vu_1),\CP(p_2, \vu_2)$ such that $\vu_1,\vu_2$ are ultraparallel.  Since $\CP(p,\vu)=\CP(p,-\vu)$, we may assume without loss of generality that $\vu_1,\vu_2$ are consistently oriented.  In that case, the crooked planes are disjoint if and only if the {\em Drumm-Goldman inequality}~\cite{DG99} holds\,:
\begin{equation}\label{DGinequality}
(p_2-p_1)\cdot\vu_1\times\vu_2>\abs{(p_2-p_1)\cdot\vu_1}+\abs{(p_2-p_1)\cdot\vu_2}.
\end{equation}

\begin{defn}
A {\em crooked foliation} is a path of pairwise disjoint crooked planes $\{ \CP(p_t, \vu_t\}_{t\in\R}$, where $ \vu_t\in\V$ is a continuous curve of spacelike vectors and $p_t\in\E$ a regular curve of points.
\end{defn}
A {\em regular curve} is a parametrized curve that is both smooth and such that its tangent vector is non-zero at all points.  Our requirement could be relaxed\,: we can easily construct foliations by crooked planes along curves that are only piecewise smooth (in fact, we only really need a $C^1$ curve) or that have zero tangent vector at a discrete set of points.  It is useful though, in order to keep the statements of Theorems~\ref{thm:tangentA},~\ref{thm:tangent} and~\ref{thm:tangentB} as ``clean'' as possible.  The curves we will consider will be orbit curves and consequently regular.  The reader will see how most of our statements might be modified to admit piecewise smooth curves.
\begin{ex}\label{basicex}
For $t\in\R$, set\,:
\begin{equation}\label{eq:normal}
\vu_t=\begin{bmatrix} \cosh(t) \\ 0 \\ \sinh(t)\end{bmatrix}.
\end{equation}
These are pairwise ultraparallel unit-spacelike vectors.  Next, set\,:
\begin{equation*}
p_t=\left(0,\alpha t, 0\right)
\end{equation*}
where $\alpha>0$.  Then for every $s>t$, $\vu_s,-\vu_t$ are consistently oriented.  The left-hand side of the Drumm-Goldman inequality yields\,:
\begin{equation*}
(p_s-p_t)\cdot(-\vu_t\times\vu_s)>\alpha(s-t)\sinh(s-t)>0.
\end{equation*}
The right-hand side is zero.  Thus every pair of distinct crooked planes is disjoint and $\{ \CP(p_t, \vu_t\}_{t\in\R}$ is a crooked foliation.  See Figure~\ref{fig:axis}.
\end{ex}
 \begin{figure}
  \includegraphics{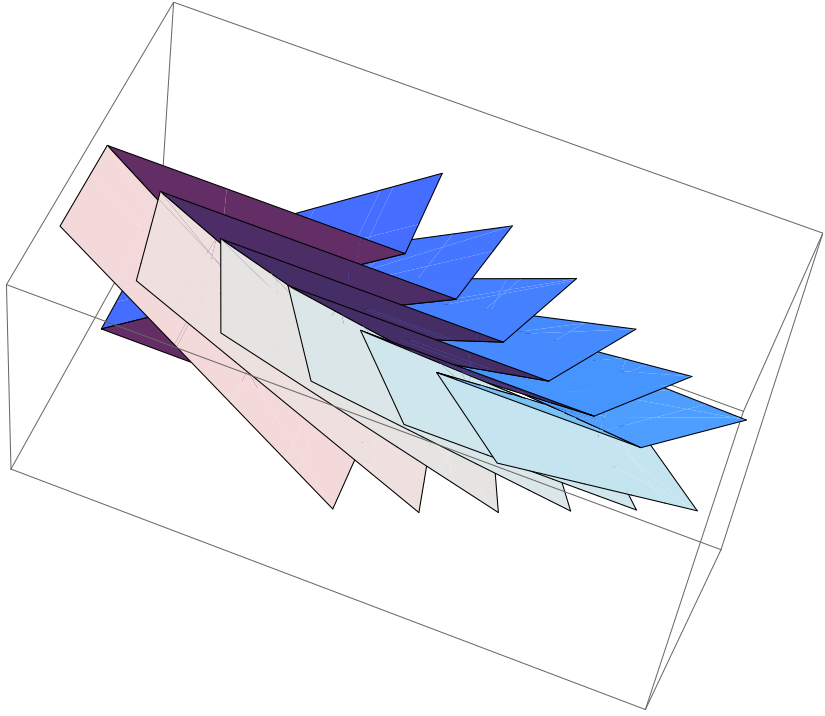}
  \caption{A crooked foliation with vertices along a line.}
  \label{fig:axis}
  \end{figure}

This example arises, in fact, from a one-parameter subgroup of Lorentzian isometries $\langle\gamma_t\rangle$, where $\L(\gamma_t)$ is hyperbolic.  Indeed, $p_t$ is the orbit of a point on the invariant axis of $\langle\gamma_t\rangle$ and $\vu_t$ is the orbit of a unit-spacelike vector, under the action of the linear part.  
%
\begin{defn}
Let $\langle \ag_t\rangle$ be a one-parameter subgroup of $\Ise$ and set $g_t=\L(\gamma_t)$.  If $\{\CP(\ag_t(p),g_t(\dv))\}$ is a crooked foliation, we say that the orbit curve $\ag_t(p)$ admits a {\em one-parameter crooked foliation}.
\end{defn}
Since disjointness of the crooked planes require $g_t(\dv)$ and $g_s(\dv)$ to be non-crossing, $g_t$ must be non-elliptic.

\section{An infinitesimal condition on disjointness}\label{sec:inf}

Starting from a variant of the original Drumm-Goldman inequality~\cite{BCDG, CDG10}, we will deduce an infinitesimal condition for disjointness, suitable for crooked foliations.  Let us underscore that the results of this section hold for all pairs of non-crossing directors of crooked planes, not just ultraparallel directors.

In keeping with the notation introduced in~\cite{BCDG}, set\,:
\begin{equation*}\label{eq:quad}
\Quad(\vu)=\{ a\xm{\vu}-b\xp{\vu}:a,b\geq0\}\setminus\{\origin\}.
\end{equation*}
It is a quadrant in $\vu^\perp\subset\V$.  Note that $\Quad(-\vu)=-\Quad(\vu)$.

\begin{thm}\cite{BCDG} \label{thm:allowable}
Suppose $\vu_1,\vu_2$ are consistently oriented, spacelike vectors.  Then the crooked planes $\CP(\vu_1,p_1)$, $\CP(\vu_2,p_2)$ are disjoint if and only if\,:
\begin{equation*}
p_2-p_1\in\inte{\Quad(\vu_2)-\Quad(\vu_1)}.
\end{equation*}
\end{thm}
Given $\vw_i\in\Quad(\vu_i)$, $i=1,2$, $\vw_2-\vw_1$ belongs to the interior of $\V(\vu_1)-\V(\vu_2)$ if either $\vw_1$ or $\vw_2$, {\em but not both}, belongs to the edge of its respective quadrant.

Now consider a continuous path of spacelike, non-crossing vectors $\vu_t$, $t\in\R$.  Since the map $t\mapsto\ldot{\vu_{0}}{\vu_t}$ is a continuous function, no pair $\vu_s$, $\vu_t$ can be consistently oriented.  That means that either $\vu_s,-\vu_t$ or $-\vu_s,\vu_t$ are consistently oriented, for $s\neq t\in\R$.  We will say that the path is {\em normalized} if for every $s>t\in\R$, $\vu_s$ and $-\vu_t$ are consistently oriented.  (It suffices to check a single pair of values $s,t$.)  For instance, $\vu_t$ described in~\eqref{eq:normal} is a normalized path.

Observe that if a continuous path of spacelike, non-crossing vectors $\vu_t$ is not normalized, then $\vu_{-t}$ is.  Thus any suitable curve of spacelike vectors can be normalized, up to changing the direction of the parametrization.

\begin{lemma}
Let $\dv_t$, $t\in\R$, be a normalized path and let $p_t$, $t\in\R$, be a regular curve in $\E$.  Then $\CP(p_t,\dv_t)$ is a crooked foliation if and only if, for every $s\neq t\in\R$ and for every $r>0$, the following four inequalities hold\,:
\begin{align}
&r(p_s-p_t) \cdot \dv_t^-{\times}\dv^+_s >0 \label{Drumm-Goldman-ineq-1} \\
&r (p_s-p_t)\cdot \dv_t^+{\times}\dv^-_s >0 \label{Drumm-Goldman-ineq-2} \\
&r(p_s-p_t)\cdot \dv_t^-{\times}\dv^-_s>0 \label{Drumm-Goldman-ineq-3} \\
& r(p_s-p_t) \cdot \dv_t^+{\times}\dv^+_s >0. \label{Drumm-Goldman-ineq-4}
\end{align}
In the asymptotic case, we omit~\eqref{Drumm-Goldman-ineq-3} if $\vu_s^-=\vu_t^-$ and ~\eqref{Drumm-Goldman-ineq-4} if $\vu_s^+=\vu_t^+$ .
\end{lemma}

\begin{proof}
Permuting the indices if necessary, we may assume that $s>t$. The vectors $-\dv_t,~\dv_s$ are then consistently oriented.  By Theorem~\ref{thm:allowable}, $\CP(p_t,\dv_t)\cap\CP(p_s,\dv_s)=\emptyset$ if and only if $p_s-p_t\in\inte{\Quad(\vu_s)+\Quad(\vu_t)}$.  Now $\Quad(\vu_s)+\Quad(\vu_t)$ is a cone, spanned by positive multiples of $\dv_s^-$, $\dv_t^-$, $-\dv_s^+$, $-\dv_t^+$.  Therefore $p_s-p_t\in\inte{\Quad(\vu_s)+\Quad(\vu_t)}$ if and only if $p_s-p_t$ lies in the intersection of the halfspaces given by~\eqref{Drumm-Goldman-ineq-1}-\eqref{Drumm-Goldman-ineq-4}.
\end{proof}
\begin{rem}
The cone $\Quad(\vu_s)+\Quad(\vu_t)$ is four-sided if $\vu_s,\vu_t$ are ultraparallel and three-sided if they are asymptotic.
\end{rem}
An infinitesimal condition for disjointness was first proposed in~\cite{BCDG}.  We prove here a slight improvement of the condition, allowing the derivative of the curve of vertices $p_t$ to belong to the edge of $\Quad(\dv_t)$.
\begin{thm}\label{thm:tangentA}
Let $\dv_t$, $t\in\R$, be a normalized path of pairwise ultraparallel spacelike vectors.  Suppose $p_t$, $t\in\R$, is a regular curve such that, for every $t\in\R$\,:
\begin{equation*}
\dot{p}_t\in\Quad(\dv_t).
\end{equation*}
Then $\CP(p_t,\dv_t)$ is a crooked foliation.
\end{thm}
\begin{proof}
Choose $t_0\in\R$ and write\,:
\begin{equation*}
p_t=p_{t_0}+\int_{t_0}^t \dot{p}_\tau d\tau.
\end{equation*}
Let $s>t\in\R$.  For every $\tau\in(t,s)$, $\Quad(\dv_\tau)$ is the convex hull of two rays, spanned by $\xm{\dv_\tau},-\xp{\dv_\tau}$, which can themselves be expressed as positive linear combinations of $\xm{\dv_t}$, $\xm{\dv_s}$, $-\xp{\dv_t}$, $-\xp{\dv_s}$.  Thus\,:
\begin{equation*}
p_s-p_t=\int_t^s \dot{p}_\tau d\tau\in\inte{\Quad(\vu_s)+\Quad(\vu_t)}.
\end{equation*}
By Theorem~\ref{thm:allowable}, the crooked planes $\CP(p_s,\dv_s)$, $\CP(p_t,\dv_t)$ are disjoint.
\end{proof}
In Example~\ref{alpha=kl}, we will illustrate the case where each $\dot{p}_t$ belongs to the edge of $\Quad(\dv_t)$; there we will compute the Drumm-Goldman inequality explicitly to show that we obtain a crooked foliation.

We now prove the converse to Theorem~\ref{thm:tangentA}.
\begin{thm}\label{thm:tangent}
Let $\dv_t$, $t\in\R$, be a continuous path of pairwise ultraparallel spacelike vectors.  Let $p_t$, $t\in\R$, be a regular curve admitting the crooked foliation $\CP(p_t,\dv_t)$.  Then for every $t\in\R$\,:
\begin{equation*}
\dot{p}_t\in\Quad(\dv_t).
\end{equation*}
\end{thm}
\begin{proof}
Substituting $t$ for $-t$ if necessary, we may suppose without loss of generality that $\dv_t$ is a normalized path.

For $s>t\in\R$, set\,:
\begin{equation*}
 v_{t,s} = \frac{1}{s-t}(p_s -p_t)
 \end{equation*}
 Then\,:
 \begin{equation*}
 \lim_{s\rightarrow t} v_{t,s}=\dot{p}_t
 \end{equation*}
Since $\vu_t$ is a positive scalar multiple of $\xm{\vu_t}\times\xp{\vu_t}$, taking the limit when $s\rightarrow t$ in~\eqref{Drumm-Goldman-ineq-1} and~\eqref{Drumm-Goldman-ineq-2} yields\,:
\begin{align*}
&\dot{p}_t \cdot \vu_t\geq 0  \\
&\dot{p}_t \cdot \vu_t\leq 0.
\end{align*}
Thus $\dot{p}_t\in\vu_t^\perp$.

Next, we will use~\eqref{Drumm-Goldman-ineq-3} to show that $\dot{p}_t\cdot\xm{\vu_t}>0$.  Conjugating if necessary, we may assume that\,:
\begin{align*}
\xm{\vu_t} & = \begin{bmatrix} 0\\ 1\\ 1\end{bmatrix} \\
\xm{\vu_s} & = \begin{bmatrix} \sech \theta(s)\\ \tanh\theta(s)\\ 1\end{bmatrix}
\end{align*}
where $\lim_{s\rightarrow t}\theta(s)=\infty$.  We may multiply $\xm{\vu_s}$ by $\cosh\theta(s)$ since it is a positive number, yielding\,:
\begin{equation*}
\cosh\theta(s)\xm{\vu_t}\times\xm{\vu_s}=\begin{bmatrix} e^{-\theta(s)} \\ 1\\ 1\end{bmatrix}.
\end{equation*}
Taking the limit of~\eqref{Drumm-Goldman-ineq-3} when $s$ goes to $t$, we obtain $\dot{p}_t\cdot\xm{\vu_t}\geq 0$.

In the same manner,~\eqref{Drumm-Goldman-ineq-4} yields $\dot{p}_t\cdot\xp{\vu_t}\leq0$.
\end{proof}

\section{One-parameter hyperbolic groups}\label{sec:hyp}

Let $g\in\SOto$ be hyperbolic.  Then its fixed eigendirection is spacelike, which we can associate to a null frame as in Definition~\ref{def:nullframe}.  Specifically, let $e^{-l}<1<e^l$ be the three eigenvalues of $g$ ($l>0$).  Let $\xpm{g}$ be the $e^{\pm l}$-eigenvector whose third coordinate is 1.  Set $\xo{g}$ to be the unique unit-spacelike 1-eigenvector such that $(\xo{g},\xm{g},\xp{g})$ is a null frame.

 Let $\langle g_t\rangle$ be a one-parameter subgroup of $\SOto$ of hyperbolic isometries, with $g_1=g$.    For every $t\in\R$\,:
\begin{align*}
\xo{g_t} & =\xo{g}\\
\xp{g_t} &= \xp{g}\\
\xm{g_t} &=\xm{g}.
\end{align*}
 Conjugating if necessary, we may suppose without loss of generality that\,:
\begin{equation}\label{normalized-hyperbolic}
g_t =
\begin{bmatrix}
 \cosh(lt) & 0 & \sinh(lt)  \\
 0& 1  & 0 \\
 \sinh(lt) & 0 &  \cosh(lt)

\end{bmatrix}
\end{equation}
where $l>0$.  Thus\,:
\begin{align*}
 \xo{g_t} & = \begin{bmatrix} 0\\ 1\\ 0\end{bmatrix} \\
 \xp{g_t} & =\begin{bmatrix} 1\\ 0\\ 1\end{bmatrix} \\
 \xm{g_t} & =\begin{bmatrix} -1\\ 0\\ 1\end{bmatrix} .
 \end{align*}
We will determine a suitable path of directors $\vu_t$ for a one-parameter crooked foliation $\{\CP(p_t,\vu_t)\}$. In order to simplify calculations, we will choose an orbit of unit-spacelike vectors in $\left(\xo{g}\right)^\perp$.  The reader can easily check that $\CP(p_t,-\vu_t)=\CP(p_t,\vu_t)$, thus one may choose either of the two unit-spacelike curves in $\left(\xo{g}\right)^\perp$.  Furthermore, we may choose the orientation such that the curve is normalized.  Therefore, set\,:
\begin{equation}\label{utcurve}
\vu_t=g_t(\vu_0)=\begin{bmatrix} \cosh(lt) \\ 0 \\ \sinh(lt)\end{bmatrix}.
\end{equation}

Let us now describe the $\langle g_t\rangle$-orbits.  First, observe that each orbit lies in a plane parallel to $\left(\xo{g}\right)^\perp$.  Set $\WP^\pm=\langle \xo{g},\xpm{g}\rangle$.  The union $\WP^+\cup\WP^-$ divides $\V$ into four sectors.  Two of these sectors contain those vectors in $(\xo{g})^\perp$ which are spacelike, and the two others, those which are timelike.  Given $\vu\in(\xo{g})^\perp$, its $g_t$-orbit is one of three types, depending on the sector to which $\vu$ belongs\,:
\begin{itemize}
\item the curve $g_t(\vu)$ is {\em spacelike} when $\vu$ is a timelike vector;
\item the curve $g_t(\vu)$ is {\em timelike} when $\vu$ is a spacelike vector;
\item the curve $g_t(\vu)$ is {\em lightlike}  when $\vu$ is a lightlike vector.
\end{itemize}

\subsection{Adding a translational part}

Let $\gamma\in\Ise$ with linear part $g$ and acting without fixed point. Let $\tp$ be the translational part of $\gamma$.  Conjugating with a translation if necessary, we may assume that $\tp=\alpha g^0$, where $\alpha\in\R$.  The value $\alpha$ is called the {\em Margulis invariant} of $\gamma$~\cite{Ma87}.

Thus $\gamma$ admits a unique invariant line $\lg$ and the restriction of $\gamma$ acts by translation\,:
\begin{equation*}
\gamma\lvert_{\lg}:x\mapsto x+\alpha \xo{g}.
\end{equation*}

Let $\gamma_t$ be the one-parameter subgroup of $\Ise$ such that $\gamma_1=\gamma$.  In particular, $\L(\gamma_t)=g_t$\,:
\begin{equation*}
\gamma_t =
\begin{bmatrix}
 \cosh(lt) & 0 & \sinh(lt) & 0 \\
0 & 1 & 0 &  \alpha t \\
 \sinh(lt) & 0 & \cosh(lt) & 0
\end{bmatrix}
\end{equation*}
where $l>0$.  The Margulis invariant of $\gamma_t$ is $\alpha\abs{t}$.

In order to obtain crooked foliations, we need disjoint crooked planes which, in turn, requires $\alpha>0$~\cite{DG99}.  We saw this in Example~\ref{basicex}.  From now on, we will assume that $\alpha>0$.  (Negative values of $\alpha$ require the use of ``negatively extended'' crooked planes; see~\cite{DG99}.)

A $\langle\gamma_t\rangle$-orbit looks like a $\langle g_t\rangle$-orbit that has been ``stretched'' in the $\xo{g}$ direction.  Let us make this statement more precise.  The planes $\WP^\pm$ admit affine counterparts, $\lg+\WP^\pm$, which divide $\E$ into four sectors.  For every $p\in\E$, there exist $q\in\lg$ and $\vx\in(\xo{g})^\perp$ such that\,:
\begin{equation}\label{eq:SandT}
p=q+\vx.
\end{equation}
Clearly, $p\in\lg+\WP^\pm$ if and only if $\vx\in\WP^\pm$, and the case $\vx=\origin$ corresponds to $p\in\lg$.  Keeping the notation in Equation~\eqref{eq:SandT}, set\,:
\begin{equation*}
\TS =\{ p\in\E : \vx\cdot\vx<0\} .
\end{equation*}
Each orbit in $\TS$, like its linear counterpart, is a spacelike curve.  Orbits in $\lg+\WP^\pm$ are now spacelike, because $\alpha\neq 0$.  To describe the remaining orbits, still keeping the notation in Equation~\eqref{eq:SandT}, for $k>0$, set\,:
\begin{equation*}
\SS_k =\{p\in\E : \vx\cdot\vx=k^2\}.
\end{equation*}
This is a hyperbolic cylinder which is, furthermore,  $\langle\gamma_t\rangle$-invariant.  As we will see later, while some of the orbits in $\SS_k$ remain timelike, the stretch factor introduced by $\alpha$ means that some of the orbits will be spacelike, depending on the value of $k$.  Finally, set\,:
\begin{equation*}
\SS=\bigcup_{k>0}\SS_k.
\end{equation*}
Alternatively\:
\begin{equation*}
\SS=\inte{\lg+\Quad(\xo{g})}\cup\inte{\lg+\Quad(-\xo{g})}.
\end{equation*}

\subsection{Orbits in $\TS\cup\WP^\pm$}\label{ssec:axis}

Example~\ref{basicex} shows that $\lg$ admits a one-parameter crooked foliation.  We will prove that this is the only orbit in $\TS\cup\WP^\pm$ to admit one.

Recall the expression for $\vu_t$ given in~\eqref{utcurve}.

First, consider an orbit in $\TS$.  Since we may place the origin anywhere along the invariant axis $\lg$, we may suppose without loss of generality that\,:
\begin{equation}\label{Torbit}
p_t=\gamma_t(p_0)=\left( k\sinh(l(t+t_0)), \alpha t, k\cosh(l(t+t_0))\right)
\end{equation}
where $k\neq 0$ and $t_0\in\R$.
\begin{lemma}
Let $p\in\TS$.  Then the orbit curve through $p$ does not admit a one-parameter crooked foliation.
\end{lemma}

\begin{proof}
By Theorem~\ref{thm:tangent}, $\dot{p}_t$ must belong to $(\vu_t)^\perp$.  However, by Equation~\eqref{Torbit}\,:
\begin{equation*}
\dot{p}_t=\left( kl\cosh l(t+t_0), \alpha , kl\sinh l(t+t_0)\right).
\end{equation*}
Therefore\,:
\begin{equation*}
\dot{p}_t\cdot\vu_t=kl\cosh(t_0)\neq 0.
\end{equation*}

\end{proof}

\begin{lemma}
Let $p\in\WP^\pm\setminus\lg$.  Then the orbit curve through $p$ does not admit a one-parameter crooked foliation.
\end{lemma}
\begin{proof}
The proof of this lemma is similar to the previous one.  Here\,:
\begin{equation}\label{W-orbit}
p_t=\gamma_t(p_0)=\left( ke^{\pm lt}, \alpha t, \pm ke^{\pm lt}\right)
\end{equation}
where $k\neq 0$.  Therefore\,:
\begin{equation*}
\dot{p}_t\cdot\vu_t= kle^{\pm lt}(\pm\cosh(lt)-\sinh(lt))\neq 0.
\end{equation*}
\end{proof}

\subsection{Orbits in $\SS$}

Again, recall the expression for $\vu_t$ in~\eqref{utcurve}.  As above, we may place the origin anywhere on the invariant line $\lg$, so that we may write an arbitrary orbit as follows\,:
\begin{equation}\label{eq:S-orbit}
p_t=\left(k\cosh(l(t+t_0)) ,\alpha t, k\sinh(l(t+t_0))\right)
\end{equation}
where $k\neq 0$ and $t_0\in\R$.  In other words, $p_t\in\SS_{\abs{k}}$.  For now, we do not assume that $k>0$, to avoid a lot of unnecessary signs.

The first condition for the orbit $p_t$ to admit a one-parameter crooked foliation is that $\dot{p}_t\in\vu_t^\perp$\,:
\begin{align*}
\dot{p}_t & = \begin{bmatrix}kl\sinh(l(t+t_0)) \\ \alpha \\ kl\cosh(l(t+t_0))\end{bmatrix}\\
\dot{p}_t\cdot\vu_t & = kl\sinh lt_0.
\end{align*}
Therefore $\dot{p}_t\in\vu_t^\perp$ if and only if $t_0=0$.  (This is what we will mean by being {\em calibrated} in~\S\ref{sec:when}.)  Next\,:
\begin{equation*}
\xpm{\vu_t}  = \begin{bmatrix} \tanh lt \\ \mp\sech lt \\ 1\end{bmatrix}.
\end{equation*}
Therefore\,:
\begin{align*}
\dot{p}_t\cdot\xm{\vu_t} & = \sech lt\left(\alpha-kl\right) \\
\dot{p}_t\cdot\xp{\vu_t} & = \sech lt\left(-\alpha-kl\right).
\end{align*}
We see here that the ratio $\alpha/l$ plays an important role, motivating the following definition.
\begin{defn}
Let $\gamma\in\Ise$ be hyperbolic with Margulis invariant $\alpha>0$.   Let $e^l$, $l>0$ be the largest eigenvalue of its linear part.  The {\em generalized Margulis invariant} of $\gamma$ is\,:
\begin{equation*}
\mu_\gamma =\frac{\alpha}{l}.
\end{equation*}
\end{defn}
If $\langle\gamma_t\rangle$ is a one-parameter hyperbolic group with positive Margulis invariants, then for every $t\in\R$\,:
\begin{equation*}
\mu_{\gamma_t}=\mu_{\gamma_1}.
\end{equation*}
Thus we may speak of the {\em generalized Margulis invariant of the one-parameter group}.
%
%

By Theorem~\ref{thm:tangentA} and its converse, Theorem~\ref{thm:tangent}, we have proved\,:
\begin{thm}\label{thm:SSorbit}
Let  $\langle\gamma_t\rangle$ be a one-parameter hyperbolic group with generalized Margulis invariant $\mu$.  Let $k>0$.  The  $\langle\gamma_t\rangle$-orbit through $p\in\SS_k$ admits a one-parameter crooked foliation if and only if\,:
\begin{equation*}
k\leq\mu.
\end{equation*}
\qed
\end{thm}

A crooked foliation along such an orbit is depicted in Figure~\ref{fig:hyp}.

\begin{figure}
  \includegraphics{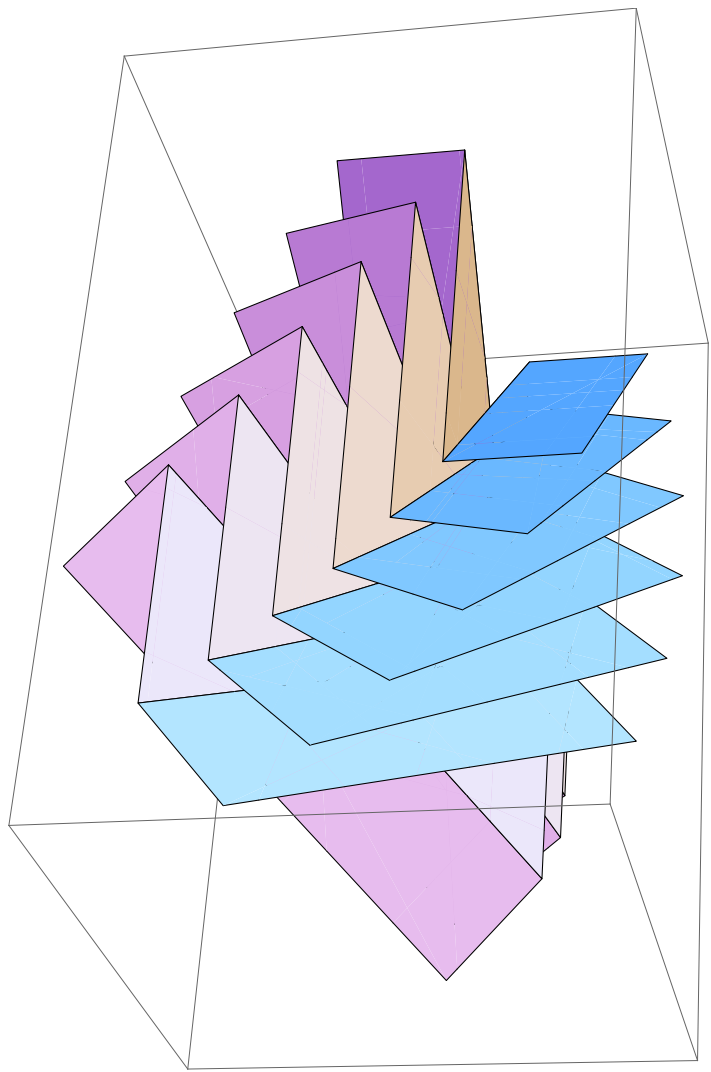}
  \caption{A one-parameter crooked foliation along an orbit in $\SS$.  The group in question is a one-parameter hyperbolic group.}
  \label{fig:hyp}
  \end{figure}

\begin{ex}\label{alpha=kl}
Applying the Drumm-Goldman inequality, we will show directly that when $k=\mu$, the crooked planes $\CP(p_s,\vu_s)$ and $\CP(p_t,\vu_t)$ are disjoint whenever $s\neq t$.

Using Equation~\eqref{eq:S-orbit} (here $k>0$)\,:
\begin{equation*}
p_t=\left(\pm k\cosh lt ,k l t, \pm k\sinh lt \right).
\end{equation*}
Thus the left-hand side of the Drumm-Goldman inequality, assuming $s>t$, is\,:
\begin{equation*}
(p_s-p_t)\cdot \vu_s\times\vu_t  = kl(s-t)\sinh l(s-t) .
\end{equation*}
The right-hand side of the Drumm-Goldman inequality is\,:
\begin{equation*}
\abs{(p_s-p_t)\cdot\vu_s}+\abs{(p_s-p_t)\cdot\vu_t} .
\end{equation*}
Evaluating the first term yields\,:
\begin{align*}
\abs{(p_s-p_t)\cdot\vu_s} & = k\abs{\cosh ls(\cosh ls-\cosh lt)-\sinh ls(\sinh ls-\sinh lt)} \\
& = k\abs{1-\cosh l(s-t)} \\
& = k\left(\cosh l(s-t)-1\right) .
\end{align*}
The calculation for the second term, $\abs{(p_s-p_t)\cdot\vu_t} $, yields the same expression.  Thus the Drumm-Goldman inequality reduces to\,:
\begin{equation*}
l(s-t)\sinh l(s-t)>2\left(\cosh l(s-t)-1\right) .
\end{equation*}
Now consider the Taylor expansions for each side\,:
\begin{align*}
x\sinh x & =x^2+\frac{x^4}{3!}+\frac{x^6}{5!}+\dots \\
2\left(\cosh x-1\right) & = x^2+\frac{1}{2}\frac{x^4}{3!}+\frac{1}{3}\frac{x^6}{5!}+\dots
\end{align*}
Thus all values of $s-t$ satisfy the Drumm-Goldman inequality.
\end{ex}

\section{One-parameter parabolic groups}\label{sec:par}

Consider now the case when $g\in\SOto$ is parabolic.  It admits a 1-dimensional fixed eigenspace, spanned by a null vector which we denote again by $\xo{g}$.  Conjugating if necessary, we may assume without loss of generality that this fixed eigenvector is\,:
\begin{equation*}
\xo{g}=\begin{bmatrix}0 \\ 1 \\ 1\end{bmatrix}.
\end{equation*}
In what follows we will consider the following basis for $\V$\,:
\begin{equation*}
\mathcal{B}=\left( \begin{bmatrix}0 \\ 1 \\ 1\end{bmatrix},\begin{bmatrix}1 \\ 0 \\ 0\end{bmatrix},\begin{bmatrix}0 \\ 2 \\ 0\end{bmatrix}\right).
\end{equation*}
The basis is positively oriented, with Gram matrix\,:
\begin{equation*}
 \left< , \right>_\mathcal{B}
   = \begin{bmatrix} 0 & 0& 2 \\ 0& 1 & 0 \\ 2 & 0 & 4 \end{bmatrix}.
\end{equation*}
The matrix of $g$ with respect to $\mathcal{B}$ is upper triangular.  More precisely, $g=g_{t_0}$, for some $t_0\in\R$, where\,:
\begin{equation*}
\left[g_t\right]_{\mathcal{B}}=
\begin{bmatrix}
1 & t & -t^2  \\
0 & 1 & -2t  \\
0 & 0 & 1
\end{bmatrix}.
\end{equation*}
Without loss of generality, assume that $g=g_1$.


Let $\gamma\in\Ise$ with linear part $g$ and a translation part $\vv$ relative to $\po=(0,0,0)$.  We seek an expression for the translational part of $\gamma^n$, $n\in\Z$, with the
ultimate goal of writing down a one-parameter subgroup $\gamma_t$ with linear part $g_t$.

If $x\in\E$\,:
\begin{equation*}
\gamma_n(x) = \po+g_n(x-\po) + (g_{n-1}+\dots +g+id)(\vv).
\end{equation*}
The matrix of $g_{n-1}+\dots +g+id$ with respect to the basis $\mathcal{B}$ is\,:
\begin{equation*}
\left[g_{n-1}+\dots +g+id\right]_{\mathcal{B}}=
 \begin{bmatrix}
 n & \frac{(n-1)n}{2} & -\frac{(n-1)n(2n-1)}{6} \\
 0 & n & -(n-1)n \\
 0 & 0 & n
 \end{bmatrix} .
\end{equation*}

Now let $\langle \gamma_t\rangle\subset\Ise$, where the linear part of $\gamma_t$ is $g_t$.  Let $x\in\E$ and consider the orbit curve $p_t=\gamma_t(x)$.  Conjugating by a translation if necessary, we may assume that $x= \po$ and therefore\,:
\begin{equation*}
 p_t  =\po+
 \begin{bmatrix}
  t & \frac{(t-1)t}{2} & -\frac{(t-1)t(2t-1)}{6} \\
   0 & t & -(t-1)t \\
   0 & 0 & t
   \end{bmatrix}
         \begin{bmatrix} a \\ b \\ c \end{bmatrix}
\end{equation*}
where $(a,b,c)$ is the translational part for $\gamma_1$, in terms of the basis $\mathcal{B}$.  Thus\,:
\begin{equation*}
\left[\dot{p}_t\right]_\mathcal{B}=\begin{bmatrix} 1 & t -\frac{1}{2} & -t^2+t-\frac{1}{6}\\ 0&1& -2t+1 \\0&0& 1 \end{bmatrix}\begin{bmatrix} a \\ b \\ c \end{bmatrix}.
\end{equation*}
Next, we determine the directors for a possible one-parameter crooked foliation.  This will be a path of spacelike vectors in $(\xo{g})^\perp$ and the following is a normalized path\,:
\begin{equation*}
\left[\vu_t\right]_\mathcal{B}=\begin{bmatrix} t\\ 1\\ 0\end{bmatrix}
\end{equation*}
with associated null frame containing the following null vectors\,:
\begin{align*}
\left[\xm{\vu}_t\right]_\mathcal{B} & = \begin{bmatrix}  1 \\ 0 \\ 0 \end{bmatrix} \\
\left[\xp{\vu}_t\right]_\mathcal{B} & = \begin{bmatrix} 1\\ \frac{2t}{t^2+1} \\  \frac{-1}{t^2+1} \end{bmatrix}.
\end{align*}
%
Recall that the infinitesimal condition for disjointness in Theorem~\ref{thm:tangentA} assumes that the directors are pairwise ultraparallel.  In the asymptotic case, $\xm{\vu_t}$ is constant and translation along the line spanned by this null vector will produce intersecting crooked planes.  However, $\xp{\vu_t}$ and $\xp{\vu_s}$ are linearly independent.  Thus the proof of Theorem~\ref{thm:tangentA} can be modified to yield\,:
\begin{thm}\label{thm:tangentB}
Let $\dv_t$, $t\in\R$, be a normalized path of pairwise asymptotic spacelike vectors with $\xm{\vu_t}$ constant.  Suppose $p_t$, $t\in\R$, is a regular curve such that, for every $t\in\R$\,:
\begin{equation*}
\dot{p}_t\in\Quad(\dv_t)\setminus\R_+\xm{\vu_t}.
\end{equation*}
Then $\CP(p_t,\dv_t)$ is a crooked foliation.
\qed
\end{thm}
Let us now apply this condition to the orbit $p_t$.
\begin{align*}
\dot{p}_t\cdot\vu_t& =\begin{bmatrix}t & 1& 0\end{bmatrix} \begin{bmatrix} 0 & 0& 2 \\ 0& 1 & 0 \\ 2 & 0 & 4 \end{bmatrix}\begin{bmatrix} 1 & t -\frac{1}{2} & -t^2+t-\frac{1}{6}\\ 0&1& -2t+1 \\0&0& 1 \end{bmatrix}\begin{bmatrix} a \\ b \\ c \end{bmatrix} \\
& =b+c.
\end{align*}
Thus $\dot{p}_t\in\vu_t^\perp$ if and only if $b=-c$.

Similar calculations yield\,:
\begin{align*}
 \dot{p}_t{\cdot}\xm{\vu}_t  & = 2c \\
  \dot{p}_t{\cdot}\xp{\vu}_t & = \frac{-2}{t^2+1}\left(a+\frac{4}{3}c\right).
\end{align*}
Therefore, $\dot{p}_t\in\Quad(\vu_t)\setminus\R_+\xm{\vu_t}$ if and only if\,:
\begin{align*}
c & > 0 \\
a & \geq -\frac{4}{3}c.
\end{align*}
Figure~\ref{fig:par} depicts such a crooked foliation.

\begin{figure}
  \includegraphics{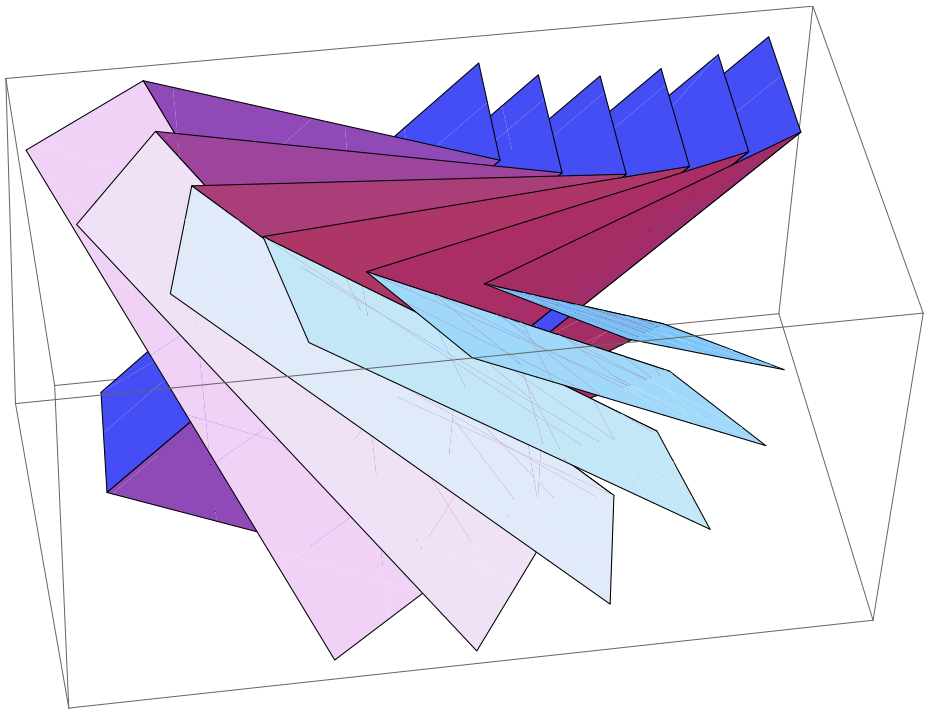}
  \caption{A one-parameter crooked foliation along an orbit in the parabolic case.}
  \label{fig:par}
  \end{figure}

\section{Remainders\,: hyperbolic+asymptotic}

One-parameter hyperbolic groups correspond quite naturally to foliations by pairwise ultraparallel crooked planes, and one-parameter parabolic groups, by pairwise asymptotic crooked planes.  Indeed,  these are the kinds of foliations which ``fill up'' the hyperbolic plane.  But one might ask, are those the only possibilities?

Take, for instance, a one-parameter parabolic subgroup $\langle\gamma_t\rangle\subset\Ise$ and let $g$ be the linear part of $\gamma_1$.  Suppose $\vu_t=\gamma_t(\vu_0)$ is a path of spacelike vectors; $\R\xo{g}$ being the unique invariant line for the $g$-action, $\vu_t^\perp\cap g(\vu_t)^\perp$ will be timelike, unless $\xo{g}\in\vu_t^\perp$.  This means that any crooked plane with director $\vu_t$ will intersect every other crooked plane in its orbit.  (The same thing happens in the hyperbolic plane.)  Therefore, no parabolic orbit curve will admit a one-parameter foliation by crooked planes with ultraparallel directors.

However, the affine setting allows us more flexibility in the hyperbolic case.  Let $\langle\gamma_t\rangle\subset\Ise$ be a one-parameter hyperbolic group with generalized Margulis invariant $\mu$.  Let the linear part $g_t$, $t\in\R$, be as in Equation~\eqref{normalized-hyperbolic} and set as before $g=g_1$.  But this time, set\,:
\begin{equation*}
\vu_t = \begin{bmatrix} e^{lt} \\ 1 \\ e^{lt}\end{bmatrix} .
\end{equation*}
Observe that $\xp{\vu_t}=\xp{g} =(1, 0,1)$ for all $t\in\R$ and\,:
\begin{equation*}
\xm{\vu_t}=\begin{bmatrix} \frac{1-e^{-2lt}}{1+e^{-2lt}} \\ \frac{2e^{-lt}}{1+e^{-2lt}} \\ 1\end{bmatrix}.
\end{equation*}
Furthermore, as $t$ goes from $-\infty$ to $\infty$, the path $\vu_t$ goes from $\xo{g}$ and asymptotically approaches the line spanned by $\xp{g}$.   This corresponds to a foliation of a halfplane in the hyperbolic plane bounded by the invariant axis for $g$.  We will display a finite set of orbit curves $p_t$ admitting a one-parameter foliation by crooked planes with asymptotic directors.

\subsection{Case 1\,: $p_t\subset l_\gamma$}

Since $\dot{p}_t=\begin{bmatrix}0\\ \alpha\\ 0\end{bmatrix}$, $\dot{p}_t\cdot\xp{\vu_t}$ is identically zero.  Thus $l_\gamma$ does not admit a one-parameter foliation by crooked planes with asymptotic directors.

\subsection{Case 2\,: $p_t\subset\WP^\pm$}

Following~\eqref{W-orbit}, $p_t\subset\WP^+$ can be written as\,:
\begin{equation*}
p_t=\left( ke^{ lt}, \alpha t, ke^{ lt}\right)
\end{equation*}
where $k\neq 0$.  But then $\dot{p}_t\cdot\vu_t=\alpha\neq 0$.  Thus a one-parameter foliation here is not possible.  On the other hand, if $p_t\in\WP^-$\,:
\begin{equation*}
p_t=\left( ke^{ -lt}, \alpha t, - ke^{- lt}\right)
\end{equation*}
where $k\neq 0$, then $\dot{p}_t\cdot\vu_t=\alpha-2kl$, which is equal to 0 if and only if $k=\frac{\mu}{2}$.  Substituting this value into the expression for $p_t$, we find\,:
\begin{align}
\dot{p}_t\cdot\xp{\vu_t} & =-{\alpha}e^{-lt} <0\\
 \dot{p}_t\cdot\xm{\vu_t} & =\frac{ \alpha}{e^{lt}+e^{-lt}}>0.
 \end{align}
Therefore, the curve\,:
\begin{equation*}
p_t=\left(\frac{\mu}{2}e^{ -lt}, \alpha t, - \frac{\mu}{2}e^{- lt}\right)
\end{equation*}
admits a one-parameter crooked foliation.

\subsection{Case 3\,: $p_t\subset\TS$}

Recall that an arbitrary orbit in $\TS$ can be written as in~\eqref{Torbit} and therefore\,:
\begin{align*}
\dot{p}_t\cdot\vu_t & = \alpha+kle^{lt}(\cosh l(t+t_0)-\sinh l(t+t_0)) \\
& = \alpha+kle^{lt-l(t+t_0)} \\
& =  \alpha+kle^{-lt_0} .
\end{align*}
Therefore $\dot{p}_t\in(\vu_t)^\perp$ if and only if $k=-\mu e^{lt_0}$.  Substituting this value into the expression for $p_t$, we verify that\,:
\begin{align*}
\dot{p}_t\cdot\xp{\vu_t} & =-\alpha e^{-lt} <0\\
 \dot{p}_t\cdot\xm{\vu_t} & = \alpha\frac{1+e^{2lt_0}}{e^{lt}+e^{-lt}}>0.
 \end{align*}
 This shows that the curve\,:
 \begin{equation*}
 p_t=\left(-\mu e^{lt_0}\sinh l(t+t_0),\alpha t, -\mu e^{lt_0}\cosh l(t+t_0)\right)
 \end{equation*}
is the unique orbit in $\TS$ admitting a one-parameter crooked foliation.

\subsection{Case 4\,: $p_t\subset\SS$}

Let $k\neq 0$ and let $p_t\in\SS_{\abs{k}}$ as in~\eqref{eq:S-orbit}.  Then $\dot{p}_t\cdot\vu_t=0$ if and only if $k=\mu e^{lt_0}$.  Thus the same calculations as above yield that the curve\,:
 \begin{equation*}
 p_t=\left(\mu e^{lt_0}\cosh l(t+t_0),\alpha t, \mu e^{lt_0}\sinh l(t+t_0)\right)
 \end{equation*}
is the unique orbit in $\SS$ admitting a one-parameter crooked foliation.

We have summarized  all the possibilities in Table 1.

\begin{table}\label{allcurves}
\begin{center}
\begin{tabular}{ |c|c|c| }\hline
                   &  ultra parallel   &  asymptotic     \\ \hline[5pt]
   hyperbolic case &  $|k| < \mu$ & very rare     \\ \hline
   parabolic  case &  impossible      & $3 a+4 c>0$, $b=-c$, $c>0$ \\[5pt] \hline
\end{tabular}
\caption{ }
\end{center}
\end{table}



%

%





\section{Existence of crooked foliations for arbitrary pairs of disjoint crooked planes}\label{sec:when}

In this last section, we will use some of the machinery developed above in the following specific situation.  Suppose that $\vu_0,\vu_1\in\V$ are a pair of unit-spacelike, ultraparallel vectors and let $p_0,p_1\in\E$ such that\,:
\begin{equation*}
\CP(p_0,\vu_0)\cap\CP(p_1,\vu_1)=\emptyset.
\end{equation*}
We will give sufficient criteria for the existence of a one-parameter crooked foliation containing the pair of crooked planes.  To do this, we will place the vertices on an orbit curve for a one-parameter hyperbolic group.  We leave to the reader to see how the arguments could be adapted to a pair of asymptotic directors, with a one-parameter parabolic group.

Since $\vu_0,\vu_0$ are ultraparallel, they span an indefinite plane, which is the orthogonal plane to a spacelike vector.  Let $\vx\in\V$ be a unit-spacelike vector such that $\vu_1,\vu_2\in\vx^\perp$ and\,:
\begin{equation*}
(p_1-p_0)\cdot\vx>0.
\end{equation*}
Observe that $(p_1-p_0)\cdot\vx\neq 0$\,: indeed, $\vx$ is parallel to $\vu_0\times\vu_1$ and the left-hand side of the Drumm-Goldman inequality must be positive.

Let $\langle g_t\rangle$ be a one-parameter subgroup of $\SOto$ with fixed eigenvector $\vx$, such that\,:
\begin{equation*}
g_1(\vu_0)=\pm\vu_1.
\end{equation*}
The map $g_t$ is conjugate to the transformation in Equation~\eqref{normalized-hyperbolic} and we are simply adjusting the value of $l$.

Finally, replacing $\vu_0$ or $\vu_1$ by their opposites if necessary -- this will not change the crooked planes -- we may assume without loss of generality that $-\vu_0,\vu_1$ are consistently oriented.  Therefore,  we may unequivocally set\,:
\begin{equation}\label{eq:ut}
\vu_t=g_t(\vu_0)
\end{equation}
where $\vu_0=g_1^{-1}(\vu_1)$.  This is a normalized curve.

The next step is to find $\gamma_t$, with linear part $g_t$, such that $\gamma_1(p_0)=p_1$.  Clearly $g_t$ must be hyperbolic, since the directors $\vu_t$ are ultraparallel.  In~\S\ref{sec:hyp}, we saw that the only possibility for a one-parameter crooked foliation is when the orbit curve is in $\SS$ or along the invariant axis; the latter case arises when $p_1-p_0$ is parallel to $\xo{g}$ and we already know that $\lg$ admits a crooked foliation, so we will focus on the case where $p_0,p_1\in\SS$.  Thus there exists $t_0\in\R$ and $k\neq 0$ such that the orbit curve in question is as in~\eqref{eq:S-orbit}\,:
\begin{equation}\label{eq:pt}
p_t=\gamma_t(p_0)=\left( k\cosh(l(t+t_0)), \alpha t, k\sinh(l(t+t_0))\right) .
\end{equation}
We stress here that the value of $l$ is determined by $\vu_0$ and $\vu_1$; because of this, we may not assume that $t_0=0$.

\begin{defn}\label{calib}
Let $\vu_0,\vu_1\in\V$ be a pair of ultraparallel unit-spacelike vectors belonging to a normalized curve.  Let $p_0,p_1\in\E$ belong to the curve~\eqref{eq:pt}.  We say that $(p_0,\vu_0)$ and $(p_1,\vu_1)$ are a {\em calibrated pair} if and only if $t_0=0$.
\end{defn}

\begin{lemma}
Let $(p_i,\vu_i)$, $i=0,1$ be as in Definition~\ref{calib}.  The pair are calibrated if and only if\,:
\begin{equation*}
\ln\left(\frac{(p_1-p_0)\cdot\xm{g}}{(p_1-p_0)\cdot\xp{g}}\right)=l.
\end{equation*}
\end{lemma}

\begin{proof}
Write $p_1-p_0=\alpha\xo{g}+k_1\xp{g}+k_2\xm{g}$, $k_1,k_2\in\R$.  Since $p_0,p_1\in\SS$, $k_1k_2>0$.  In particular\,:
\begin{align*}
(p_1-p_0)\cdot\xm{g} & =k_1\xp{g}\cdot\xm{g} \\
(p_1-p_0)\cdot\xp{g} & =k_2\xp{g}\cdot\xm{g}.
\end{align*}
A simple calculation shows that\,:
\begin{equation*}
k=\pm\frac{2\sqrt{k_1k_2}}{e^{l/2}-e^{-l/2}}.
\end{equation*}
To simplify the argument, we will consider the case where $k>0$.

The first component of $p_1-p_0$ in the standard basis is $k_1-k_2$ and the third component is $k_1+k_2$.  Therefore\,:
\begin{equation*}
2k_1=ke^{lt_0}(e^l-1).
\end{equation*}
Therefore\,:
\begin{equation*}
t_0=-\frac{1}{2l}\left[ \ln\frac{k_1}{k_2}-l\right]
\end{equation*}
and the result follows.
\end{proof}
%
\begin{thm}\label{1-2foliate}
Let  $(p_0,\vu_0)$ and $(p_1,\vu_1)$ be a calibrated pair and suppose that
$\CP(p_0,\vu_0)$, $\CP(p_1,\vu_1)$ are disjoint crooked planes.  Then there exists a crooked foliation containing them.
\end{thm}
\begin{proof}
Let $\langle\gamma_t\rangle$ be the one-parameter subgroup generated by $\gamma_1$ described above.  The condition for a one-parameter crooked foliation along the orbit $\gamma_t(p_0)$, assuming that $t_0=0$, is independent of $t$, as the calculations leading to Theorem~\ref{thm:SSorbit} show.  Therefore, if $\CP(p_0,\vu_0)$, $\CP(p_1,\vu_1)$ are disjoint, then so are every pair of crooked planes along the same orbit.
\end{proof}

\begin{rem}
The condition that we have a calibrated pair is necessary for this argument in order to use Theorem~\ref{thm:SSorbit}; if $t_0\neq 0$, we will get intersecting crooked planes for small differences in $t$, even if the original pair is disjoint.
\end{rem}

 \bibliographystyle{amsplain}
 \bibliography{Vref}

\end{document}